\documentclass{article}
\usepackage{amssymb,amsbsy}
\usepackage{amsmath}
\usepackage{amsthm}
\usepackage{mathrsfs}
\usepackage[mathcal]{euscript}
\usepackage{color}
\usepackage[dvipsnames]{xcolor}
\usepackage[hidelinks]{hyperref}
\usepackage{tikz}
\usepackage[shortlabels]{enumitem}
\usepackage{float}

\usepackage[utf8]{inputenc}
\usepackage{fontenc}
\usepackage{array} 
\usepackage{geometry}
\newtheorem{theorem}{Theorem}[section]
\newtheorem{lemma}[theorem]{Lemma}
\newtheorem{proposition}[theorem]{Proposition}
\newtheorem{corollary}[theorem]{Corollary}

\newtheorem{claim}[theorem]{Claim}

\theoremstyle{definition}

\newtheorem{example}[theorem]{Example}
\newtheorem{algorithm}[theorem]{Algorithm}

\theoremstyle{remark}
\newtheorem{remark}[theorem]{Remark}

\begin{document}
	
	\title{The Frobenius number corresponding to the squares of three consecutive Fibonacci numbers: comparison of three algorithmic processes}
	
	\author{Aureliano M. Robles-P{\'e}rez\thanks{Departamento de Matem{\'a}tica Aplicada \& Instituto de Matem{\'a}ticas (IMAG), Universidad de Granada, 18071 Granada, Spain. \newline E-mail: \textbf{arobles@ugr.es} (\textit{corresponding author}); ORCID: \textbf{0000-0003-2596-1249}.} \\
		\mbox{ and} Jos{\'e} Carlos Rosales\thanks{Departamento de {\'A}lgebra \& Instituto de Matem{\'a}ticas (IMAG), Universidad de Granada, 18071 Granada, Spain. \newline E-mail: \textbf{jrosales@ugr.es}; ORCID: \textbf{0000-0003-3353-4335}.} }
	
	\date{\today}
	
	\maketitle
	
	\begin{abstract}
		We compute the Frobenius number for numerical semigroups generated by the squares of three consecutive Fibonacci numbers. We achieve this by using and comparing three distinct algorithmic approaches: those developed by Ram{\'i}rez Alfons{\'i}n and R{\o}dseth (\cite{ram_alf-rod}), Rosales and Garc{\'i}a-S{\'a}nchez (\cite{ros-gar_san}), and Tripathi (\cite{tripathi}).
	\end{abstract}
	
	\medskip
	

	\section{Introduction}
	
	The Frobenius problem, also known as the Frobenius coin or the Chicken McNugget problem, is a classical problem in additive number theory. Given a finite set of positive integers $A=\{a_1,\dots,a_e\}$ with $\gcd(a_1,\dots,a_e)=1$, the problem is to compute the largest integer that cannot be represented as a non-negative integer linear combination of $a_1,\dots,a_e$ (see \cite{brauer}). This largest integer is known as the \textit{Frobenius number of the set $A$}, denoted by $\mathrm{F}(A)$.
	
	An explicit formula for $\mathrm{F}(A)$ is well-known when $e=2$. Specifically, Sylvester first published the formula, $\mathrm{F}(\{a_1,a_2\})=a_1a_2-a_1-a_2$, in 1883 (see \cite{sylvester}). However, the Frobenius problem remains open for $e \geq 3$,  as no general explicit formula is known. Indeed, in \cite{curtis}, Curtis demonstrated that a polynomial formula (that is, a finite set of polynomials) cannot be used to compute the Frobenius number for $e=3$. Furthermore, in \cite{alfonsinNP}, Ram{\'i}rez Alfons{\'i}n proved that this question is NP-hard if $e$ is part of the input. Consequently, significant efforts have been dedicated to obtaining partial results or developing algorithms to determine the Frobenius number in such cases (see \cite{alfonsin} for further details). In particular, numerous papers investigate the Frobenius problem for sequences $\{a_1,\dots,a_e\}$ that belong to classical integer sequences, such as arithmetic and almost arithmetic (\cite{brauer,roberts,lewin,selmer}), geometric (\cite{ong-ponomarenko}), squares and cubes (\cite{squares-cubes,moscariello}), among others.
	
	The Fibonacci sequence is given by the recurrence relation $f_n=f_{n-1}+f_{n-2}$, for $n\geq 2$, with initial conditions $f_0=0$ and $f_1=1$ (see \url{https://oeis.org/A000045}). In this work, we are interested in the Frobenius problem for the squares of three consecutive Fibonacci numbers; that is, we aim to compute $\mathrm{F}(f_n^2, f_{n+1}^2, f_{n+2}^2)$ for $n\in\mathbb{N}$.
	
	This work focuses mainly on comparing the algorithms developed respectively by Ram{\'i}rez Alfons{\'i}n and R{\o}dseth (\cite{ram_alf-rod}), Tripathi (\cite{tripathi}), and Rosales and Garc{\'i}a-S{\'a}nchez (\cite{ros-gar_san}). It can be argued that the algorithms by Ram{\'i}rez Alfons{\'i}n and R{\o}dseth, and by Tripathi, offer a more direct and simpler approach for individual numerical semigroups. On the contrary, when searching for general formulas for a family of numerical semigroups, as we propose, the algorithm by Rosales and Garc{\'i}a-S{\'a}nchez will be the more appropriate.
	
	This work is structured as follows. In Section~\ref{preliminar}, we present foundational concepts of numerical semigroups (which are closely related to the Frobenius problem) and, in particular, the Ap{\'e}ry set of a numerical semigroup. In Sections~\ref{rar}, \ref{tri}, and \ref{rgs}, we apply algorithms developed respectively by Ram{\'i}rez Alfons{\'i}n and R{\o}dseth, by Tripathi, and by Rosales and Garc{\'i}a-S{\'a}nchez to the proposed problem. It should be noted that, in Sections~\ref{rar} and \ref{tri}, several results are presented as ``Claim'' because we have not been able to provide rigorous proofs for them. However, we are confident in the correctness of these results because, in addition to computational evidence, we obtain the same answer to the problem as in Section~\ref{rgs}, where all the results are formally proved. Observe that, in Subsection~\ref{tri-pre}, we present alternative definitions and statements to those found in \cite{tripathi}, due to identified mathematical errors in the original article (as also pointed out in \cite{suhajda}).

	\section{Preliminaries}\label{preliminar}
	
	In this work, we use numerical semigroups, which are a specific type of semigroup (or monoid) intrinsically linked to the Frobenius problem.
	
	Let $(\mathbb{N},+)$ be the additive monoid of non-negative integers. The submonoid $M$ generated by a set of positive integers $A=\{a_1,\dots,a_e\}$ is defined as $M=\langle A \rangle =\langle a_1,\dots,a_e \rangle = \{ k_1a_1+\dots+k_ea_e \mid k_i\in\mathbb{N} \mbox{ for all } i \}$, where $A$ is referred to as a \textit{system of generators of $M$}. Furthermore, if $M$ has a finite complement in $\mathbb{N}$, then we say that $M$ is a \textit{numerical semigroup}. This allows for a reformulation of the Frobenius problem: determining the largest integer that is not an element of $M$. Such an integer is the \textit{Frobenius number of $S$} and is denoted by $\mathrm{F}(M)$.
	
	Recall that $\langle A \rangle$ is a numerical semigroup if and only if $\gcd(A)=1$ (see Lemma~2.1 of \cite{springer}).
	
	Let $A$ be a system of generators for the submonoid $M$ of $(\mathbb{N},+)$. If $M\not=\langle B \rangle$ for any proper subset $B\subsetneq A$, then $A$ is called the \textit{minimal system of generators of $M$}. Corollary~2.8 of \cite{springer} demonstrates that each submonoid of $(\mathbb{N},+)$ possesses a unique and finite minimal system of generators. We denote by $\mathrm{msg}(M)$ the minimal system of generators of $M$, and the cardinality of $\mathrm{msg}(M)$, denoted by $\mathrm{e}(M)$, is the \textit{embedding dimension of $M$}.
	
	If $S$ is a numerical semigroup and $m\in S\setminus\{0\}$, then the \textit{Ap{\'e}ry set of $m$ in $S$} (named in honour of \cite{apery}) is the set $\mathrm{Ap}(S,m)=\{s\in S \mid s-m\not\in S\}$. This is a useful set for describing the numerical semigroup $S$ and solving the Frobenius problem. In particular, we will use the \textit{Ap{\'e}ry set of $S$} defined by $\mathrm{Ap}(S) = \mathrm{Ap}(S,\mathrm{m}(S))$, where $\mathrm{m}(S)=\min{(S\setminus \{0\})}$ is called the \textit{multiplicity of $S$}.
	
	From Lemma~2.4 of \cite{springer} and Lemma 3 of \cite{brauer-shockley}, we derive the following result.
	
	\begin{proposition}\label{prop01}
		Let $S$ be a numerical semigroup and $m\in S\setminus\{0\}$. Then the cardinality of $\mathrm{Ap}(S,m)$ is $m$ and $\mathrm{Ap}(S,m)=\{w(0)=0, w(1), \dots, w(m-1)\}$, where $w(i)$ is the least element of $S$ congruent with $i$ modulo $m$. Moreover, $\mathrm{F}(S)=\max(\mathrm{Ap}(S,m))-m$.
	\end{proposition}
	
	Without loss of generality, we can consider the generators of a numerical semigroup to be pairwise coprime. This is supported by expression (7) in \cite{ram_alf-rod} or Lemma 2.16 in \cite{springer}, from which we have the following result.
	
	\begin{lemma}
		Let $S$ be a numerical semigroup generated by $\{a_1,\dots,a_e\}$. Let $d=\gcd\{a_1,\dots,a_{e-1}\}$ and set $T=\langle \frac{a_1}{d},\dots,\frac{a_{e-1}}{d},a_e \rangle$. Then $\mathrm{Ap}(S,a_e)=d(\mathrm{Ap}(T,a_e))$.
	\end{lemma}
	
	From the above lemma, we deduce the following generalization of a result by Johnson (\cite{johnson}).
	
	\begin{corollary}
		Let $S$ be a numerical semigroup generated by $\{a_1,\dots,a_e\}$. Let $d=\gcd\{a_1,\dots,a_{e-1}\}$ and set $T=\langle \frac{a_1}{d},\dots,\frac{a_{e-1}}{d},a_e \rangle$. Then $\mathrm{Ap}(S) = d\mathrm{Ap}(T) + (d-1)a_e$.
	\end{corollary}
	
	From now on, we are interested in numerical semigroups generated by pairwise coprime generators. Under this assumption, the algorithms by Ram{\'i}rez Alfons{\'i}n and R{\o}dseth, by Tripathi (\cite{tripathi}), and by Rosales and Garc{\'i}a-S{\'a}nchez become applicable.
	
	\begin{remark}
		In \cite{ros-gar_san} and \cite{tripathi}, it is explicitly stated that the generators must be pairwise coprime. Thus, it is not certain that the algorithms of those works can be applied in general cases (that is, $S=\langle a_1, a_2, a_3 \rangle$ where only $\gcd(a_1, a_2, a_3)=1$ is assumed). For example, and contrary to what is remarked in \cite{sury}, it is not possible to directly apply Tripathy's results to compute $\mathrm{F}(T_n,T_{n+1},T_{n+2})$, where $T_n$ is the triangular number $\binom{n+1}{2}$ (see \cite{rpr}).
	\end{remark}
	
	Given the consecutive Fibonacci numbers $f_n$, $f_{n+1}$, and $f_{n+2}$, for $n\in\mathbb{N}$, we define the numerical semigroup $S(n)=\langle f_n^2, f_{n+1}^2, f_{n+2}^2\rangle$.
	
	Since $1$ belongs to $S(0)$, $S(1)$, and $S(2)$, it follows that $S(0) = S(1) = S(2) = \mathbb{N}$, and thus the Frobenius number in these cases is $-1$. Furthermore, $S(3)=\langle4,9,25\rangle = \langle4,9\rangle$. Consequently, in this instance, the Frobenius number can be computed using Sylvester's solution, which yields $\mathrm{F}(S(3)) = \mathrm{F}(\{4,9\})=4\times9-4-9=23$. For the remainder of this work, we will consider $n\geq4$.
	
	Observe that $\gcd(f_n,f_{n+1})=\gcd(f_n,f_n+f_{n-1})=\gcd(f_n,f_{n-1})$. By induction, it can therefore be prove that $\gcd(f_n,f_{n+1})=1$ for all $n\geq1$ (and thus for $n\geq4$). Moreover, since $\gcd(f_n^2,f_{n+1}^2)=\left(\gcd(f_n,f_{n+1})\right)^2$, we can conclude that $\gcd(f_n^2,f_{n+1}^2)=1$ for all $n\geq1$. Consequently, the elements of $\{f_n^2, f_{n+1}^2, f_{n+2}^2\}$ are pairwise coprime and, in particular, $\gcd(f_n^2,f_{n+1}^2,f_{n+2}^2)=1$ for all $n\geq1$ (and thus for $n\geq4$).
	
	Let us now demonstrate that $f_{n+2}^2$ cannot be expressed as a non-negative integer linear combination of $f_n^2$ and $f_{n+1}^2$ (for all $n\geq 4$). We proceed by \textit{reductio ad absurdum}, supposing there exist $\lambda,\mu\in\mathbb{N}$ such that $f_{n+2}^2=\lambda f_n^2+\mu f_{n+1}^2$. We distinguish three cases.
	\begin{enumerate}
		\item If $\lambda=0$, then $f_{n+2}^2=\mu f_{n+1}^2$, which implies $f_{n+1}^2|f_{n+2}^2$. This contradicts $\gcd(f_{n+1}^2,f_{n+2}^2)=1$.
		\item If $\mu=0$, then $f_{n+2}^2=\lambda f_n^2$. Thus, $(f_n+f_{n+1})^2=f_n^2+2f_nf_{n+1}+f_{n+1}^2=\lambda f_n^2$ and consequently $f_{n}\mid f_{n+1}^2$. This again leads to a contradiction.
		\item If $\lambda\not=0$ and $\mu\not=0$, then $(f_n+f_{n+1})^2=\lambda f_n^2+\mu f_{n+1}^2$ and then $2f_n f_{n+1} = (\lambda-1) f_n^2+(\mu-1) f_{n+1}^2$ with $\lambda-1,\mu-1\in\mathbb{N}$. Since $\gcd(f_n,f_{n+1})=1$, then $f_n\mid(\mu-1)$ and $f_{n+1}\mid(\lambda-1)$. Therefore, $2=\frac{\lambda-1}{f_{n+1}}f_n + \frac{\mu-1}{f_n}f_{n+1}$ with $\frac{\lambda-1}{f_{n+1}},\frac{\mu-1}{f_n}\in\mathbb{N}$, which is a contradiction.
	\end{enumerate}
	
	The preceding two paragraphs lead to the following result.
	
	\begin{lemma}\label{lem02}
		If $n\geq4$, then $S(n)$ is a numerical semigroup with $\mathrm{e}(S(n))=3$.
	\end{lemma}
	
	Let $S$ be a numerical semigroup. Following the notation introduced in \cite{JPAA}, an integer $x$ is called a \textit{pseudo-Frobenius number of $S$} if $x\in \mathbb{Z}\setminus S$ and $x+s\in S$ for all $s\in S\setminus\{0\}$. We denote by $\mathrm{PF}(S)$ the set of all the pseudo-Frobenius numbers of $S$. The cardinality of $\mathrm{PF}(S)$ is a notable invariant of $S$ (see \cite{barucci}), which is called \textit{type of $S$}, denoted by $\mathrm{t}(S)$.
	
 	Let us define the following binary relation over $\mathbb{Z}$: $a\leq_S b$ if $b-a\in S$. As stated in \cite{springer}, it is clear that $\leq_S$ is a non-strict partial order relation (that is, reflexive, transitive, and anti-symmetric). Proposition~7 of \cite{froberg} (also presented as Proposition~2.20 of \cite{springer}) characterises pseudo-Frobenius numbers in terms of the maximal elements of $\mathrm{Ap}(S,m)$ with respect to the relation $\leq_S$.
 	
	\begin{proposition}\label{prop03}
		Let $S$ be a numerical semigroup and $m\in S\setminus\{0\}$. Then 
		\[\mathrm{PF}(S)=\{w-m \mid w\in \mathrm{Maximals}_{\leq_S} (\mathrm{Ap}(S,m))  \}.\]
	\end{proposition}
	
	From Proposition~10.21 of \cite{springer}, we have the following result.
	
	\begin{proposition}\label{prop04}
		If $n\geq4$, then $\mathrm{t}(S(n))=2$.
	\end{proposition}

	\section{Algorithmic process by Ram{\'i}rez Alfons{\'i}n and R{\o}dseth}\label{rar}
	
	Let us define the numerical semigroup $S=\langle a,b,c\rangle$, where $a<b<c$ are pairwise relatively prime positive integers and $\mathrm{e}(S)=3$. We set $s_{-1}=a$ and $s_0$ such that $(b-a)s_0 \equiv c \pmod{s_{-1}}$, which simplifies to $bs_0 \equiv c \pmod{a}$.
	
	The following \textit{Algorithm Ap{\'e}ry} from \cite{ram_alf-rod} has been adapted for the purposes of this study.
	
	\begin{algorithm}\label{alg01rar}
		\mbox{} \\
		\textbf{Input:} $a,\, b,\, c,\, s_0$. \\
		\textbf{Output:} $s_v,\, s_{v+1},\, P_v,\, P_{v+1}$. \\
		\textbf{1.} $r_{-1} = a,\, r_0 = s_0$. \\
		\textbf{2.} $r_{i-1} = \kappa_{i+1}r_i +r_{i+1},\, \kappa_{i+1} = \lfloor\frac{r_{i-1}}{r_i}\rfloor$,\, $0 = r_{\mu+1} < r_\mu < \dots < r_{-1}$. \\
		\textbf{3.} $p_{i+1} = \kappa_{i+1}p_i +p_{i-1},\, p_{-1} = 0,\, p_0 = 1$. \\
		\textbf{4.} $T_{i+1} = -\kappa_{i+1}T_i + T_{i-1},\, T_{-1} = b,\, T_0 = \frac{1}{a}(br_0-c)$. \\
		\textbf{5.} IF there is a minimal $u$ such that $T_{2u+2} \leq 0$, THEN
		\[ \begin{pmatrix} s_v & P_v \\ s_{v+1} & P_{v+1} \end{pmatrix} = \begin{pmatrix} \gamma & 1 \\ \gamma_{-1} & 1 \end{pmatrix} \begin{pmatrix} r_{2u+1} & -p_{2u+1} \\ r_{2u+2} & p_{2u+2} \end{pmatrix},\quad
		\gamma = \left\lfloor\frac{-T_{2u+2}}{T_{2u+1}}\right\rfloor+1, \]
		\textbf{6.} ELSE $s_v = r_\mu,\, s_{v+1} = 0,\, P_v = p_\mu,\, P_{v+1} = p_{\mu+1}$.
	\end{algorithm}
	
	Consider the sets
	\[ A = \{ (y,z) \in \mathbb{Z}^2 \mid 0 \leq y < s_v-s_{v+1},\, 0 \leq z < P_{v+1} \} \]
	and
	\[ B = \{ (y,z) \in \mathbb{Z}^2 \mid 0 \leq y < s_v,\, 0 \leq z < P_{v+1}-P_v \}. \]
	
	\begin{remark} 
		Algorithm Ap{\'e}ry is, as the authors mention in \cite{ram_alf-rod}, a result of adapting Greenberg's ideas (\cite{greenberg}) to an earlier algorithm proposed by R{\o}dseth (\cite{rodseth}). Consequently, Algorithm~\ref{alg01rar} can be considered a rewriting of the algorithm given by Greenberg for numerical semigroups with embedding dimension equal to three.
	\end{remark}
	
	\begin{theorem}\label{thm02rar}
		Let $S$ be the numerical semigroup $S=\langle a,b,c\rangle$, with $a<b<c$ pairwise relatively prime positive integers, such that $\mathrm{e}(S)=3$. Then
		\[ \mathrm{Ap}(S, a) = \{ by+cz \mid (y,z) \in A \cup B \}. \]
		Moreover, $\mathrm{F}(S) = \max\{bs_v+c(P_{v+1}-P_v), b(s_v-s_{v+1})+cP_{v+1} \}-a-b-c$.
	\end{theorem}
	
	\begin{remark}\label{rem03rar}
		More precisely, $\mathrm{PF}(S) = \{bs_v+c(P_{v+1}-P_v)-a-b-c, b(s_v-s_{v+1})+cP_{v+1}-a-b-c \}$.
	\end{remark}
	
	Next, we will apply Algorithm~\ref{alg01rar} to the numerical semigroups $S(n)=\langle f_n^2, f_{n+1}^2, f_{n+2}^2 \rangle$ with $n\geq 4$ (and $f_n$, $f_{n+1}$, $f_{n+2}$ three consecutive Fibonacci numbers).
	
	\begin{lemma}\label{lem04rar}
		We have that
		\[ s_0=\begin{cases} f_nf_{n-3}+1, \;\; \mbox{if $n$ is even,} \\ 2f_nf_{n-2}+1, \;\; \mbox{if $n$ is odd,}	\end{cases} \quad
		T_0=\begin{cases} f_nf_{n-1}-2, \;\; \mbox{if $n$ is even,} \\ 2f_n^2-3, \;\; \mbox{if $n$ is odd.} \end{cases} \]
	\end{lemma}
	
	\begin{proof}
		Let us suppose that $n$ is even. Then
		\begin{align*}
			bs_0-c & = f_{n+1}^2(f_nf_{n-3}+1) - f_{n+2}^2 \\
			& = f_n(f_{n+1}^2f_{n-3} - f_{n+3}) \\
			& = f_n^2(f_nf_{n-3} + 2f_{n-1}f_{n-3}) + f_n(f_{n-1}^2f_{n-3} - f_{n+3})  \\
			& = f_n^2(f_nf_{n-3} + 2f_{n-1}f_{n-3}) + f_n(f_{n-1}^2f_{n-3} - 3f_n - 2f_{n-1}) \\
			& = f_n^2(f_nf_{n-3} + 2f_{n-1}f_{n-3} - 3) + f_nf_{n-1}(f_{n-1}f_{n-3} - 2) \\
			& = f_n^2(f_nf_{n-3} + 2f_{n-1}f_{n-3} - 3 + f_{n-1}f_{n-4}) \\
			& = f_n^2(f_nf_{n-1} - 2) \\
			& = a(f_nf_{n-1} - 2).
		\end{align*}
		Since $f_nf_{n-3}+1<f_n^2$, the result follows.
		
		Now let us assume that $n$ is odd. Then
		\begin{align*}
			bs_0-c & = f_{n+1}^2(2f_nf_{n-2}+1) - f_{n+2}^2 \\
			& = f_n(2f_{n+1}^2f_{n-2} - f_{n+3}) \\
			& = f_n^2(2f_nf_{n-2} + 4f_{n-1}f_{n-2}) + f_n(2f_{n-1}^2f_{n-2} - f_{n+3}) \\
			& = f_n^2(2f_nf_{n-2} + 4f_{n-1}f_{n-2}) + f_n(2f_{n-1}^2f_{n-2} - 3f_n - 2f_{n-1}) \\
			& = f_n^2(2f_nf_{n-2} + 4f_{n-1}f_{n-2} - 3) + 2f_nf_{n-1}(f_{n-1}f_{n-2} - 1) \\
			& = f_n^2(2f_nf_{n-2} + 4f_{n-1}f_{n-2} - 3 + 2f_{n-1}f_{n-3}) \\
			& = f_n^2(2f_n^2 - 3) \\
			& = a(2f_n^2 - 3).
		\end{align*}
		Since $2f_nf_{n-2}+1<f_n^2$, the result is obtained.
	\end{proof}
	
	\begin{claim}\label{claim05rar}
		Let $n\geq4$.
		\begin{enumerate}
			\item If $n=6n_0+4$, with $n_0\in\mathbb{N}$, then
			\begin{itemize}
				\item $\mu = 4n_0+1$,
				\item $\kappa_1=\dots=\kappa_{2n_0}=4$, $\kappa_{2n_0+1}=2$, $\kappa_{2n_0+2}=\dots=\kappa_{4n_0+2}=4$,
				\item $p_{\mu+1}=f_n^2$,
				\item $T_{\mu+1}=-f_{n+2}^2$,
				\item $u=n_0$,
				\item $T_{2u+1}=\frac{f_{n+5}}{2}$, $T_{2u+2}=-\frac{3f_{n+5}+2f_{n+3}}{2}$,
				\item $\begin{pmatrix} r_{2u+1} & -p_{2u+1} \\ r_{2u+2} & p_{2u+2} \end{pmatrix} =
				\begin{pmatrix} \frac{f_{n-1}}{2} & -\frac{f_n+f_{n-2}}{2} \\ \frac{f_{n-4}}{2} & \frac{5f_n+3f_{n-2}}{2} \end{pmatrix}$.
			\end{itemize}
			
			\item If $n=6n_0+5$, with $n_0\in\mathbb{N}$, then
			\begin{itemize}
				\item $\mu = 4n_0+2$,
				\item $\kappa_1=1$, $\kappa_2=5$, $\kappa_3=4$ ($n_0=0$),
				\item $\kappa_1=1$, $\kappa_2=3$, $\kappa_3=\dots=\kappa_{2n_0+1}=4$, $\kappa_{2n_0+2}=4$, $\kappa_{2n_0+3}=\dots=\kappa_{4n_0+3}=4$ ($n_0\geq1$),
				\item $p_{\mu+1}=f_n^2$,
				\item $T_{\mu+1}=-f_{n+2}^2$,
				\item $u=n_0$,
				\item $T_{2u+1}=\frac{f_{n+4}}{2}$, $T_{2u+2}=-\frac{f_{n+5}+f_{n+3}}{2}$,
				\item $\begin{pmatrix} r_{2u+1} & -p_{2u+1} \\ r_{2u+2} & p_{2u+2} \end{pmatrix} =
				\begin{pmatrix} \frac{f_{n+1}}{2} & -\frac{f_{n-2}}{2} \\ \frac{f_{n-2}}{2} & \frac{2f_{n-1}+3f_{n-2}}{2} \end{pmatrix}$.
			\end{itemize}
			 
			\item If $n=6n_0+6$, with $n_0\in\mathbb{N}$, then
			\begin{itemize}
				\item $\mu = 4n_0+3$,
				\item $\kappa_1=\dots=\kappa_{2n_0}=4$, $\kappa_{2n_0+1}=3$, $\kappa_{2n_0+2}=1$, $\kappa_{2n_0+3}=3$, $\kappa_{2n_0+4}=\dots=\kappa_{4n_0+4}=4$,
				\item $p_{\mu+1}=f_n^2$,
				\item $T_{\mu+1}=-f_{n+2}^2$,
				\item $u=n_0$,
				\item $T_{2u+1}=f_{n+4}$, $T_{2u+2}=-\frac{f_{n+3}}{2}$,
				\item $\begin{pmatrix} r_{2u+1} & -p_{2u+1} \\ r_{2u+2} & p_{2u+2} \end{pmatrix} =
				\begin{pmatrix} f_{n+1} & -f_{n-2} \\ \frac{f_n}{2} & \frac{f_n}{2} \end{pmatrix}$.
			\end{itemize}
			
			\item If $n=6n_0+1$, with $n_0\in\mathbb{N}\setminus\{0\}$, then
			\begin{itemize}
				\item $\mu = 4n_0$,
				\item $\kappa_1=1$, $\kappa_2=3$, $\kappa_3=\dots=\kappa_{2n_0}=4$, $\kappa_{2n_0+1}=2$, $\kappa_{2n_0+2}=\dots=\kappa_{4n_0+1}=4$,
				\item $p_{\mu+1}=f_n^2$,
				\item $T_{\mu+1}=-f_{n+2}^2$,
				\item $u=n_0$,
				\item $T_{2u+1}=\frac{f_{n+5}}{2}$, $T_{2u+2}=-\frac{3f_{n+5}+2f_{n+3}}{2}$,
				\item $\begin{pmatrix} r_{2u+1} & -p_{2u+1} \\ r_{2u+2} & p_{2u+2} \end{pmatrix} =
				\begin{pmatrix} \frac{f_{n-1}}{2} & -\frac{f_n+f_{n-2}}{2} \\ \frac{f_{n-4}}{2} & \frac{5f_n+3f_{n-2}}{2} \end{pmatrix}$.
			\end{itemize}
			
			\item If $n=6n_0+8$, with $n_0\in\mathbb{N}$, then
			\begin{itemize}
				\item $\mu = 4n_0+3$,
				\item $\kappa_1=\dots=\kappa_{2n_0+1}=4$, $\kappa_{2n_0+2}=6$, $\kappa_{2n_0+3}=\dots=\kappa_{4n_0+4}=4$,
				\item $p_{\mu+1}=f_n^2$,
				\item $T_{\mu+1}=-f_{n+2}^2$,
				\item $u=n_0$,
				\item $T_{2u+1}=\frac{f_{n+4}}{2}$, $T_{2u+2}=-\frac{f_{n+6}-f_{n+2}}{2}$,
				\item $\begin{pmatrix} r_{2u+1} & -p_{2u+1} \\ r_{2u+2} & p_{2u+2} \end{pmatrix} =
				\begin{pmatrix} \frac{f_{n+1}}{2} & -\frac{f_{n-2}}{2} \\ \frac{f_{n-2}}{2} & \frac{2f_{n-1}+3f_{n-2}}{2} \end{pmatrix}$.
			\end{itemize} 
			
			\item If $n=6n_0+3$, with $n_0\in\mathbb{N}\setminus\{0\}$, then
			\begin{itemize}
				\item $\mu = 4n_0+2$,
				\item $\kappa_1=1$, $\kappa_2=3$, $\kappa_3=\dots=\kappa_{2n_0}=4$, $\kappa_{2n_0+1}=3$, $\kappa_{2n_0+2}=2$, $\kappa_{2n_0+3}=3$, $\kappa_{2n_0+4}=\dots=\kappa_{4n_0+3}=4$,
				\item $p_{\mu+1}=f_n^2$,
				\item $T_{\mu+1}=-f_{n+2}^2$,
				\item $u=n_0$,
				\item $T_{2u+1}=f_{n+4}$, $T_{2u+2}=-\frac{f_{n+3}}{2}$,
				\item $\begin{pmatrix} r_{2u+1} & -p_{2u+1} \\ r_{2u+2} & p_{2u+2} \end{pmatrix} =
				\begin{pmatrix} f_{n+1} & -f_{n-2} \\ \frac{f_n}{2} & \frac{f_n}{2} \end{pmatrix}$.
			\end{itemize}
		\end{enumerate}
	\end{claim}
	
	\begin{corollary}\label{cor06rar}
		For $n\geq 4$, let $S(n)$ be the numerical semigroup generated by the squares of the three consecutive Fibonacci numbers $f_n$, $f_{n+1}$, $f_{n+2}$.
		\begin{enumerate}
			\item If $n=6n_0+4$ or $n=6n_0+7$ with $n_0\in\mathbb{N}$ (equivalently, if $n=3k_0+1$ with $k_0\in\mathbb{N}\setminus\{0\}$), then
			\begin{itemize}
				\item $\gamma=4$,
				\item $\begin{pmatrix} s_v & P_v \\ s_{v+1} & P_{v+1} \end{pmatrix} =
				\begin{pmatrix} \frac{f_{n+2}}{2} & \frac{f_{n-1}}{2} \\ f_n & f_n \end{pmatrix}$,
				\item $\mathrm{PF}(S(n)) = \left\{ \left(\frac{f_{n+2}}{2}-1\right)f_{n+1}^2 + \left(\frac{f_{n-2}+f_n}{2}-1\right)f_{n+2}^2, \left( \frac{f_{n-1}}{2}-1\right)f_{n+1}^2 + (f_n-1)f_{n+2}^2 \right\} - f_n^2$.
			\end{itemize}
			
			\item If $n=6n_0+5$ or $n=6n_0+8$ with $n_0\in\mathbb{N}$ (equivalently, if $n=3k_0+2$ with $k_0\in\mathbb{N}\setminus\{0\}$), then
			\begin{itemize}
				\item $\gamma=3$,
				\item $\begin{pmatrix} s_v & P_v \\ s_{v+1} & P_{v+1} \end{pmatrix} =
				\begin{pmatrix} f_{n+2} & f_{n-1} \\ \frac{f_n+f_{n+2}}{2} & \frac{f_{n+1}}{2} \end{pmatrix}$,
				\item $\mathrm{PF}(S(n)) = \left\{ (f_{n+2}-1)f_{n+1}^2 + \left(\frac{f_{n-2}}{2}-1\right)f_{n+2}^2, \left(\frac{f_{n+1}}{2}-1\right)f_{n+1}^2 + \left(\frac{f_{n+1}}{2}-1\right)f_{n+2}^2 \right\} - f_n^2$.
			\end{itemize}
						
			\item If $n=6n_0+6$ or $n=6n_0+9$, with $n_0\in\mathbb{N}$ (equivalently, if $n=3k_0+3$ with $k_0\in\mathbb{N}\setminus\{0\}$), then
			\begin{itemize}
				\item $\gamma=1$,
				\item $\begin{pmatrix} s_v & P_v \\ s_{v+1} & P_{v+1} \end{pmatrix} =
				\begin{pmatrix} \frac{f_{n+3}}{2} & \frac{f_{n-3}}{2} \\ \frac{f_n}{2} & \frac{f_n}{2} \end{pmatrix}$,
				\item $\mathrm{PF}(S(n)) = \left\{ \left(\frac{f_{n+3}}{2}-1\right)f_{n+1}^2 + (f_{n-2}-1)f_{n+2}^2, (f_{n+1}-1)f_{n+1}^2 + \left(\frac{f_n}{2}-1\right)f_{n+2}^2 \right\} - f_n^2$.
			\end{itemize}
		\end{enumerate}
		Moreover, in any case, $\mathrm{F}(S(n)) = \max(\mathrm{PF}(S(n)))$.
	\end{corollary}

	\section{Algorithmic process by Tripathi}\label{tri}
	
	Although \cite{tripathi} discusses ``formulae for the Frobenius number in three variables'', we opt for the term algorithmic process. This is because deriving the coefficients for Tripathi's proposed formulas necessitates the use of computationally non-trivial preceding algorithms, as demonstrated in Subsection~\ref{tri-app}.
	
	\subsection{Preliminaries}\label{tri-pre}
	
	Once again, let $S$ be the numerical semigroup generated by the pairwise relatively prime positive integers $a<b<c$ such that $\mathrm{e}(S)=3$.
	
	We define $\kappa:=\left\lfloor \frac{c}{b} \right\rfloor$ and $\ell :\equiv cb^{-1} \pmod{a}$ (where $\ell$ is the smallest non-negative integer satisfying the congruence).
	
	From the proof of Lemma~2 in \cite{tripathi}, we have the following result.
	
	\begin{lemma}\label{lem01tri}
		Let $a<b<c$ be pairwise relatively prime positive integers such that $S=\langle a,b,c\rangle$ is a numerical semigroup with $\mathrm{e}(S)=3$. Then $\ell > \kappa$.
	\end{lemma}
	
	We define $q:=\left\lfloor \frac{a}{a-\ell} \right\rfloor$, $r:=a-q(a-\ell)$, $\overline{q}:=\left\lfloor \frac{a}{\ell} \right\rfloor$, and $\overline{r}:=a-\overline{q}\ell$, (with $q,\overline{q}\geq1$, $0\leq r<a-\ell$, and $0\leq \overline{r}<\ell$).
	
	The following results are Theorem~3 and Theorem~4 from \cite{tripathi}, respectively.
	
	\begin{theorem}\label{thm02tri3}
		If $\ell > \kappa$ and $br<cq$, then
		\[ \mathrm{F}(\langle a,b,c \rangle) = \begin{cases} b\{ (\lambda+1)(a-\ell)+r-1 \}-a, \mbox{ if } \lambda \geq \frac{c(q-1)-br}{b(a-\ell)+c}, \\[3pt]
			b(a-\ell-1) + c(q-\lambda-1)-a, \mbox{ if } \lambda \leq \frac{c(q-1)-br}{b(a-\ell)+c}, \end{cases} \]
		where $\lambda:= \left\lfloor \frac{c(q-1)-br}{b(a-\ell)+c} \right\rfloor$.
	\end{theorem}
	
	\begin{theorem}\label{thm03tri4}
		If $\ell > \kappa$ and $b(\ell-\overline{r}) < c(\overline{q}-1)$, then
		\[ \mathrm{F}(\langle a,b,c \rangle) = \begin{cases} b(\ell-1) + c(\overline{q}-1) - a, \mbox{ if } 0 \leq \overline{r} < \ell-\kappa, \\
			b(\overline{r}-1) + c\overline{q} - a, \mbox{ if } \ell-\kappa \leq \overline{r} < \ell. \end{cases} \]
	\end{theorem}
	
	We define $A:=br-cq$ and $B:=b(a-\ell-r)+c(q+1)$. Furthermore, let $\mu'$ be the largest non-negative integer $m$ such that $\left\lfloor \frac{Bi}{A} \right\rfloor = \left\lfloor \frac{(a-\ell-r)i}{r} \right\rfloor$ for all $i\in\{0,\dots,m\}$.
	
	\begin{remark}\label{rem04tri}
		 To avoid ambiguity, the definition of $\mu'$ has been slightly modified from that in \cite{tripathi}: we define $\mu'$ as the largest non-negative integer $m$ such that $\left\lfloor \frac{Bm}{A} \right\rfloor = \left\lfloor \frac{(a-\ell-r)m}{r} \right\rfloor$. By this definition, in Example~5 of \cite{tripathi}, we would get $\mu'=9$. However, the correct value is $\mu'=4$, which agrees with our definition and, moreover, is the one given in \cite{tripathi}. This fact was also pointed out in Remark~5.17 of \cite{suhajda}.
	\end{remark}
	
	The following claim will be necessary for Theorem~\ref{thm08tri5}. Set $\mathcal{X} = \{ x_i \mid 0\leq i \leq \mu' \}$, where $x_i=r-\left((a-\ell)i\pmod{r}\right)$.
	
	\begin{claim}\label{claim05tri}
		Let $\ell > \kappa$ and $br > cq$. Let $\mathcal{X}_{ord}$ be the set formed by the ordered elements of $\mathcal{X}$, and let $\mathcal{X}_{dist}$ be the set formed by the distances between two consecutive elements of $\mathcal{X}_{ord}$. If $\mu'> \left\lfloor \frac{r}{u} \right\rfloor$, then $\mathcal{X}_{dist}$ has exactly two elements, $d_1<d_2$.		
	\end{claim}
	
	\begin{remark}\label{rem06tri}
		In Theorem~5 of \cite{tripathi}, $d_1 = \left\lceil \frac{r}{u} \right\rceil u-r$ is defined. However, this equality is incorrect, as demonstrated in Example~5.18 of \cite{suhajda}.
	\end{remark}
	
	\begin{remark}\label{rem07tri}
		In \cite{suhajda}, it is conjectured that $d_2$ always corresponds to $\mu'$. We have preferred to be cautious and define $d_2$ as the second possible value in $\mathcal{X}_{dist}$. Moreover, we believe that it will be easier to prove Claim~\ref{claim05tri} than Theorem~5.16 in \cite{suhajda}.
	\end{remark}
	
	The following result is a reformulation of Theorem~5 in \cite{tripathi} that takes into account the above two remarks.
	
	\begin{theorem}\label{thm08tri5}
		Let $\ell > \kappa$ and $br > cq$. Let $u \equiv (a-\ell)\pmod{r}$.
		\begin{itemize}
			\item[(a)] If $\mu'< \left\lfloor \frac{r}{u} \right\rfloor$, then $\mathrm{F}(\langle a,b,c \rangle) =$ 
			\[ \max\left\{ b(r-\mu'u-1), b(u-1)+c\left(\mu' + \left( \left\lfloor \frac{(a-\ell)\mu'}{r} \right\rfloor+1 \right)q \right) \right\} + cq\left\lfloor \frac{a-\ell-1}{r} \right\rfloor -a. \]
			
			\item[(b)] Let $\mu'> \left\lfloor \frac{r}{u} \right\rfloor$. Let $d_1$, $d_2$ be the numbers given by Claim~\ref{claim05tri}, and let $p_i$ be the largest positive integer such that $x_{p_i}+d_i\in\mathcal{X}$ for $i=1,2$. Then
			\[ \mathrm{F}(\langle a,b,c \rangle) = \max\{b(d_1-1)+cy_{p_1}, b(d_2-1)+cy_{p_2} \} + cq\left\lfloor \frac{a-\ell-1}{r} \right\rfloor -a, \]
			where $y_{p_i} = q\left(\left\lfloor \frac{(a-\ell)p_i}{r} \right\rfloor +1\right) + p_i$ for $i=1,2$.
		\end{itemize} 
	\end{theorem}
	
	\begin{remark}\label{rem09tri}
		Regarding Claim~\ref{claim10tri} and Theorem~\ref{thm11tri6}, similar remarks to Remarks~\ref{rem04tri}, \ref{rem06tri}, and \ref{rem07tri} can be made.  
	\end{remark}
	
	We define $\overline{A}:=b(\ell-\overline{r})-c(\overline{q}-1)$ and $\overline{B}=b\overline{r}+c\overline{q}$. Furthermore, let $\overline{\mu}'$ be the largest non-negative integer $m$ such that $\left\lfloor \frac{\overline{B}i}{\overline{A}} \right\rfloor = \left\lfloor \frac{\overline{r}i}{\ell-\overline{r}} \right\rfloor$ for all $i\in\{0,\dots,m\}$.
	
	Set $\overline{\mathcal{X}} = \{ x_i \mid 0\leq i \leq \overline{\mu}' \}$, where $x_i=(\ell-\overline{r})-\left(\overline{r}i\pmod{(\ell-\overline{r})}\right)$.
	
	\begin{claim}\label{claim10tri}
		Let $\ell > \kappa$ and $b(\ell-\overline{r}) > c(\overline{q}+1)$. Let $\overline{\mathcal{X}}_{ord}$ be the set formed by the ordered elements of $\overline{\mathcal{X}}$, and let $\overline{\mathcal{X}}_{dist}$ be the set formed by the distances between two consecutive elements of $\overline{\mathcal{X}}_{ord}$. If $\overline{\mu}'> \left\lfloor \frac{\ell-\overline{r}}{\overline{u}} \right\rfloor$, then $\overline{\mathcal{X}}_{dist}$ has exactly two elements $d_1<d_2$.		
	\end{claim}
	
	\begin{theorem}\label{thm11tri6}
		Let $\ell > \kappa$ and $b(\ell-\overline{r}) > c(\overline{q}+1)$. Let $\overline{u} \equiv \overline{r}\pmod{(\ell-\overline{r})}$.
		\begin{itemize}
			\item[(a)] If $\overline{\mu}'< \left\lfloor \frac{\ell-\overline{r}}{\overline{u}} \right\rfloor$, then $\mathrm{F}(\langle a,b,c \rangle) =$ 
			\[ \max\left\{ b(\ell-\overline{r}-\overline{\mu}'\overline{u}-1), b(\overline{u}-1)+c\left( \left(\left\lfloor \frac{\ell\overline{\mu}'}{\ell-\overline{r}} \right\rfloor+1 \right)\overline{q} + \left\lfloor \frac{\overline{r}\overline{\mu}'}{\ell-\overline{r}} \right\rfloor+1 \right) \right\} + \]
			\[ c\left((\overline{q}+1)\left\lfloor \frac{\ell-1}{\ell-\overline{r}} \right\rfloor-2\right) -a. \]
			
			\item[(b)] Let $\overline{\mu}'> \left\lfloor \frac{\ell-\overline{r}}{\overline{u}} \right\rfloor$. Let $d_1$, $d_2$ be the numbers given by Claim~\ref{claim10tri}, and let $p_i$ be the largest positive integer such that $x_{p_i}+d_i\in\overline{\mathcal{X}}$ for $i=1,2$. Then
			\[ \mathrm{F}(\langle a,b,c \rangle) = \max\{b(d_1-1)+cy_{p_1}, b(d_2-1)+cy_{p_2} \} + c\left((\overline{q}+1)\left\lfloor \frac{\ell-1}{\ell-\overline{r}} \right\rfloor-2\right) -a, \]
			where $y_{p_i} = \overline{q}\left(\left\lfloor \frac{\ell p_i}{\ell-\overline{r}} \right\rfloor +1 \right) + \left\lfloor \frac{\overline{r} p_i}{\ell-\overline{r}} \right\rfloor+1$ for $i=1,2$.
		\end{itemize} 
	\end{theorem}

	\subsection{Application to $S(n)$}\label{tri-app}
	
	Let us now apply the results from the previous subsection to $S(n)=\langle f_n^2, f_{n+1}^2, f_{n+2}^2 \rangle$ for $n\geq 4$, where $f_n$, $f_{n+1}$, $f_{n+2}$ are three consecutive Fibonacci numbers.
	
	\begin{lemma}\label{lem12tri}
		If $n\geq 4$, then $\kappa = 2$ and
		\[ \ell = \begin{cases} f_nf_{n-3}+1 \mbox{ if $n$ is even,} \\[3pt] 2f_nf_{n-2}+1 \mbox{ if $n$ is odd.}
		\end{cases} \]
	\end{lemma}
	
	\begin{proof}
		For $\kappa$, we have that
		\[ 2 < \frac{f_{n+2}^2}{f_{n+1}^2} < 3 \Longleftrightarrow 2 < \frac{f_{n+2}^2 - f_{n+1}^2}{f_{n+1}^2} + 1 < 3 \Longleftrightarrow f_{n+1}^2 < 2f_{n+1}f_n + f_n^2 < 2f_{n+1}^2. \]
		Observe that
		\begin{itemize}
			\item the left inequality $f_{n+1}^2 < 2f_{n+1}f_n + f_n^2\,$ is true because $f_{n+1} < 2f_n$,
			\item and the right inequality $2f_{n+1}f_n + f_n^2 < 2f_{n+1}^2$ is equivalent to $f_n^2 < 2f_{n+1}f_{n-1}$, which is true because $f_n < 2f_{n-1}$.
		\end{itemize}
		Therefore, $\kappa=2$.
		
		Now, for $\ell$, recalling that $f_{n+1}f_{n-1}-f_n^2=(-1)^n$ (Cassini's identity), we derive that
		\[ f_{n+1}^2f_{n-1}^2 = (f_{n+1}f_{n-1})^2 = (f_n^2+(-1)^n)^2 = f_n^4 + 2(-1)^nf_n^2 +1 \equiv 1 \pmod{f_n^2} \Longrightarrow b^{-1}=f_{n-1}^2. \]
		Then
		\[ cb^{-1} = f_{n+2}^2f_{n-1}^2 = f_{n+1}^2f_{n-1}^2 + 2f_{n+1}f^nf_{n-1}^2 + f_n^2f_{n-1}^2 = \]
		\[ (f_n^2+(-1)^n)^2 +2(f_n^2+(-1)^n)f_nf_{n-1} + f_n^2f_{n-1}^2 = \]
		\[ f_n^4 + 2(-1)^nf_n^2 + 1 +2f_n^3f_{n-1} + 2(-1)^nf_nf_{n-1} + f_n^2f_{n-1}^2 \equiv (2(-1)^nf_nf_{n-1}+1) \pmod{f_n^2}. \]
		Thus, if $n$ is even, then
		\[ \ell = 2f_nf_{n-1}+1 - f_n^2 = f_nf_{n-3}+1\]
		and, if $n$ is odd,
		\[ \ell = -2f_nf_{n-1}+1 + 2f_n^2 = 2f_nf_{n-2}+1.\]
	\end{proof}
	
	\begin{lemma}\label{lem13tri}
		If $n$ is even, then $q=1$, $r=\ell$,
		\[ \overline{q} = \begin{cases} 2, \mbox{ if } n=4, \\ 3, \mbox{ if } n=6, \\ 4, \mbox{ if } n\geq 8, \end{cases} \]
		and
		\[ \overline{r} = \begin{cases} 1, \mbox{ if } n=4, \\ 13, \mbox{ if } n=6, \\ f_nf_{n-6}-4, \mbox{ if } n\geq 8. \end{cases}  \]
		
		If $n$ is odd, then
		\[ q = \begin{cases} 6, \mbox{ if } n=5, \\ 4, \mbox{ if } n\geq 7, \end{cases} \]
		\[ r = \begin{cases} 1, \mbox{ if } n=5, \\ f_nf_{n-6}+4, \mbox{ if } n\geq 7, \end{cases}  \]
		$\overline{q}=1$, and $\overline{r}=f_nf_{n-3}-1$.
	\end{lemma}
	
	\begin{proof}
		Let us suppose that $n$ is even. For $n=4$ and $n=6$, we check the equalities by direct computation. Therefore, let us assume that $n\geq 8$.
		
		For $q$, we need to prove that
		\[1 < \frac{f_n^2}{f_n^2-f_nf_{n-3}-1} < 2,\]
		which is equivalent to
		\[ f_nf_{n-3}+1 < 2f_nf_{n-3}+2 < f_n^2. \]
		The left inequality is trivial, and the right one follows from the equality $f_n=3f_{n-3}+2f_{n-4}$.
		
		Now, since $q=1$, then $r=\ell$ from their definitions.
		
		For $\overline{q}$, we need to prove that
		\[ 4 < \frac{f_n^2}{f_nf_{n-3}+1} < 5, \]
		which (using again the equality $f_n=3f_{n-3}+2f_{n-4}$) is equivalent to
		\[ f_nf_{n-3}+4 < 2f_nf_{n-4} < 2f_nf_{n-3}+5. \]
		Here, the right inequality is trivial, and the left one follows from the inequality $4<f_nf_{n-6}$.
		
		Finally, $\overline{r} = f_n^2 - 4(f_nf_{n-3}-1) = f_n(f_n-4f_{n-3}) + 4 = f_nf_{n-6} + 4$.
		
		Using arguments similar to those already presented, the result follows for odd $n$.
	\end{proof}
	
	\begin{corollary}\label{cor14tri}
		Let $n\geq4$.
		\begin{itemize}
			\item If $n$ is even, then $br>cq$, and therefore we can apply Theorem~\ref{thm08tri5}.
			\item If $n=5$, then $br<cq$, and therefore we can apply Theorem~\ref{thm02tri3}.
			\item If $n\geq7$ is odd, then $br>cq$, and therefore we can apply Theorem~\ref{thm08tri5}.
		\end{itemize}
	\end{corollary}
	
	\begin{proof}
		If $n\geq4$ is even, then we have that
		\[ br>cq \Longleftrightarrow f_{n+1}^2(f_nf_{n-3}+1) > f_{n+2}^2\cdot1 \Longleftrightarrow f_{n+1}^2f_{n-3} > f_{n+3} \Longleftrightarrow f_{n+1}f_{n-3} > 2 + \frac{f_n}{f_{n+1}}, \]
		which is true since $f_{n+1}f_{n-3} > 3$ for all $n\geq4$.
		
		For $n=5$, we have that $br=64\cdot1<169\cdot6=cq$.
		
		Finally, if $n\geq7$ is odd, then
		\[ br>cq \Longleftrightarrow f_{n+1}^2(f_nf_{n-6}+4) > f_{n+2}^2\cdot1 \Longleftrightarrow f_nf_{n-6}+3 > 2\frac{f_n}{f_{n+1}} + \left(\frac{f_n}{f_{n+1}}\right)^2, \]
		which is true since $f_nf_{n-6}>0$ and $3 > 2\frac{f_n}{f_{n+1}} + \left(\frac{f_n}{f_{n+1}}\right)^2$ for all $n\geq7$.
	\end{proof}
	
	\begin{lemma}\label{lem15tri}
		If $n$ is even, then
		\[ u = \begin{cases} 1, \mbox{ if } n=4, \\ 13, \mbox{ if } n=6, \\ f_nf_{n-6}-4, \mbox{ if } n\geq 8.
		\end{cases}\]
		
		If $n\geq 7$ is odd, then
		\[ u = \begin{cases} 4, \mbox{ if } n=7, \\ 55, \mbox{ if } n=9, \\ f_nf_{n-9}-17, \mbox{ if } n\geq 11.
		\end{cases}\]
	\end{lemma}
	
	\begin{proof}
		If $n\in\{4,6,7,9\}$, then we deduce the result by direct calculation.
		
		If $n\geq8$ is even, then
		\[ f_n^2 - (f_nf_{n-3}+1) = 4f_nf_{n-3} + f_nf_{n-6} - (f_nf_{n-3}+1) = 3(f_nf_{n-3}+1) +(f_nf_{n-6} - 4), \]
		and the result follows from $0 < f_nf_{n-6} - 4 < f_nf_{n-3}+1$.
		
		If $n\geq11$ is odd, then
		\[ f_n^2 - (2f_nf_{n-2}+1) = f_nf_{n-3} - 1 = 4(f_nf_{n-6}+4) +(f_nf_{n-9} - 17). \]
		Since $0 < f_nf_{n-9} - 17 < f_nf_{n-6}+4$, the result is proved.
	\end{proof}
	
	\begin{claim}\label{claim16tri}
		If $n\geq 4$ is even, then
		\[ \mu' = \begin{cases} f_{n-3}-1, \mbox{ if } n\equiv 1 \pmod{3}, \\ \frac{f_{n-3}}{2}-1, \mbox{ if } n\equiv 0 \pmod{3}, \\ \frac{f_{n-2}}{2}-1, \mbox{ if } n\equiv 2 \pmod{3}.
		\end{cases}\]
		
		If $n\geq 7$ is odd, then
		\[ \mu' = \begin{cases} f_{n-6}-1, \mbox{ if } n\equiv 1 \pmod{3}, \\ \frac{f_{n-6}}{2}-1, \mbox{ if } n\equiv 0 \pmod{3}, \\ \frac{f_{n-5}}{2}-1, \mbox{ if } n\equiv 2 \pmod{3}.
		\end{cases}\]
	\end{claim}
	
	\begin{corollary}\label{cor17tri}
		It holds that:
		\begin{itemize}
			\item If $n\in\{4,6,8\}$, then $\mu'<\left\lfloor \frac{r}{u} \right\rfloor$, and therefore we can apply Theorem~\ref{thm08tri5}(a).
			\item If $n\geq 10$ is even, then $\mu'>\left\lfloor \frac{r}{u} \right\rfloor$, and therefore we can apply Theorem~\ref{thm08tri5}(b).
			\item If $n\in\{7,9,11\}$, then $\mu'<\left\lfloor \frac{r}{u} \right\rfloor$, and therefore we can apply Theorem~\ref{thm08tri5}(a).
			\item If $n\geq13$ is odd, then $\mu'>\left\lfloor \frac{r}{u} \right\rfloor$, and therefore we can apply Theorem~\ref{thm08tri5}(b).
		\end{itemize}
	\end{corollary}
	
	\begin{proof}
		If $n\in\{4,6,7,8,9,11\}$, we deduce the result by direct calculation.
		
		For $n\geq10$ even, we have that $\mu'>13$. Moreover,
		\[ \left\lfloor \frac{r}{u} \right\rfloor = \left\lfloor \frac{f_nf_{n-3}+1}{f_nf_{n-6}-4} \right\rfloor = 3+\left\lfloor \frac{2f_nf_{n-7}+13}{f_nf_{n-6}-4} \right\rfloor. \]
		Now,
		\[ 10 > \frac{2f_nf_{n-7}+13}{f_nf_{n-6}-4} \Longleftrightarrow 10f_nf_{n-6} > 2f_nf_{n-7}+53, \]
		which is true since $8f_nf_{n-6} > 53$ for all $n\geq10$.		
		
		For $n\geq13$ odd, we have that $\mu'>12$. Moreover,
		\[ \left\lfloor \frac{r}{u} \right\rfloor = \left\lfloor \frac{f_nf_{n-6}+4}{f_nf_{n-9}-17} \right\rfloor = 3+\left\lfloor \frac{2f_nf_{n-10}+55}{f_nf_{n-9}-17} \right\rfloor. \]
		Now,
		\[ 9 > \frac{2f_nf_{n-10}+55}{f_nf_{n-9}-17} \Longleftrightarrow 9f_nf_{n-9} > 2f_nf_{n-10}+208 \]
		which is true since $7f_nf_{n-9} > 208$ for all $n\geq13$.
	\end{proof}
	
	\begin{claim}\label{claim18tri}
		If $n\geq 10$ is even, then the applicable values in Theorem~\ref{thm08tri5}(b) are
		\[ \begin{cases} d_1 = \frac{f_{n-1}}{2},\, d_2 = \frac{f_{n+2}}{2},\, p_1 = f_{n-3}-1,\, p_2 = \frac{f_{n-3}+f_{n-5}}{2}-1, \mbox{ if } n\equiv 1 \pmod{3}, \\[3pt] d_1 = f_{n+1},\, d_2 = \frac{f_{n+3}}{2},\, p_1 = \frac{f_{n-3}}{2}-1,\, p_2 = f_{n-5}-1, \mbox{ if } n\equiv 0 \pmod{3}, \\[3pt] d_1 = \frac{f_{n+1}}{2},\, d_2 = f_{n+2},\, p_1 = \frac{f_{n-2}}{2}-1,\, p_2 = \frac{f_{n-5}}{2}-1, \mbox{ if } n\equiv 2 \pmod{3}.
		\end{cases}\]
		
		If $n\geq 13$ is odd, then the applicable values in Theorem~\ref{thm08tri5}(b) are
		\[ \begin{cases} d_1 = \frac{f_{n-1}}{2},\, d_2 = \frac{f_{n+2}}{2},\, p_1 = f_{n-6}-1,\, p_2 = \frac{f_{n-6}+f_{n-8}}{2}-1, \mbox{ if } n\equiv 1 \pmod{3}, \\[3pt] d_1 = f_{n+1},\, d_2 = \frac{f_{n+3}}{2},\, p_1 = \frac{f_{n-6}}{2}-1,\, p_2 = f_{n-8}-1, \mbox{ if } n\equiv 0 \pmod{3}, \\[3pt] d_1 = \frac{f_{n+1}}{2},\, d_2 = f_{n+2},\, p_1 = \frac{f_{n-5}}{2}-1,\, p_2 = \frac{f_{n-8}}{2}-1, \mbox{ if } n\equiv 2 \pmod{3}.
		\end{cases}\]
	\end{claim}
	
	\begin{remark}\label{rem19tri}
		We have that $p_1=\mu'$.
	\end{remark}
	
	\begin{corollary}\label{cor20tri}
		Regarding Theorem~\ref{thm08tri5}(b), it holds that:
		\begin{enumerate}
			\item If $n=3k_0+1$ with $k_0\in\mathbb{N}\setminus\{0\}$, then
			\begin{itemize}
				\item $y_{p_1}=f_n-4$, $y_{p_2}=\frac{f_n+f_{n-2}}{2}-4$, if $n\geq10$ is even,
				\item $y_{p_1}=f_n-17$, $y_{p_2}=\frac{f_n+f_{n-2}}{2}-17$, if $n\geq13$ is odd,
				\item $\mathrm{PF}(S(n)) = \left\{ \left( \frac{f_{n-1}}{2}-1\right)f_{n+1}^2 + (f_n-1)f_{n+2}^2, \left(\frac{f_{n+2}}{2}-1\right)f_{n+1}^2 + \left(\frac{f_n+f_{n-2}}{2}-1\right)f_{n+2}^2 \right\} - f_n^2$.
			\end{itemize}

			\item If $n=3k_0+3$ with $k_0\in\mathbb{N}\setminus\{0\}$, then
			\begin{itemize}
				\item $y_{p_1}=\frac{f_n}{2}-4$, $y_{p_2}=f_{n-2}-4$, if $n\geq12$ is even,
				\item $y_{p_1}=\frac{f_n}{2}-17$, $y_{p_2}=f_{n-2}-17$, if $n\geq15$ is odd,
				\item $\mathrm{PF}(S(n)) = \left\{ (f_{n+1}-1)f_{n+1}^2 + \left(\frac{f_n}{2}-1\right)f_{n+2}^2, \left(\frac{f_{n+3}}{2}-1\right)f_{n+1}^2 + (f_{n-2}-1)f_{n+2}^2 \right\} - f_n^2$.
			\end{itemize}
			
			\item If $n=3k_0+2$ with $k_0\in\mathbb{N}\setminus\{0\}$, then
			\begin{itemize}
				\item $y_{p_1}=\frac{f_{n+1}}{2}-4$, $y_{p_2}=\frac{f_{n-2}}{2}-4$, if $n\geq14$ is even,
				\item $y_{p_1}=\frac{f_{n+1}}{2}-17$, $y_{p_2}=\frac{f_{n-2}}{2}-17$, if $n\geq17$ is odd,
				\item $\mathrm{PF}(S(n)) = \left\{ \left(\frac{f_{n+1}}{2}-1\right)f_{n+1}^2 + \left(\frac{f_{n+1}}{2}-1\right)f_{n+2}^2, (f_{n+2}-1)f_{n+1}^2 + \left(\frac{f_{n-2}}{2}-1\right)f_{n+2}^2 \right\} - f_n^2$.
			\end{itemize}
		\end{enumerate}
		Moreover, in any case, $\mathrm{F}(S(n)) = \max(\mathrm{PF}(S(n)))$.
	\end{corollary}
	
	With arguments similar to those in Corollary~\ref{cor14tri}, we establish the following result.
	
	\begin{corollary}\label{cor21tri}
		Let $n\geq4$.
		\begin{itemize}
			\item If $n\in\{4,6\}$, then $b(\ell-\overline{r})<c(\overline{q}+1)$, and therefore we can apply Theorem~\ref{thm03tri4}.
			\item If $n\geq8$ is even, then $b(\ell-\overline{r})>c(\overline{q}+1)$, and therefore we can apply Theorem~\ref{thm11tri6}.
			\item If $n\geq5$ is odd, then $b(\ell-\overline{r})>c(\overline{q}+1)$, and therefore we can apply Theorem~\ref{thm11tri6}.
		\end{itemize}
	\end{corollary}
	
	With arguments similar to those in Lemma~\ref{lem15tri}, we establish the following result.
	
	\begin{lemma}\label{lem22tri}
		If $n\geq8$ is even, then $\overline{u} = \overline{r} =f_nf_{n-6}-4$.
		
		If $n\geq 5$ is odd, then $\overline{u} = \overline{r} =f_nf_{n-3}-1$.
	\end{lemma}
	
	\begin{claim}\label{claim23tri}
		If $n\geq 8$ is even, then
		\[ \overline{\mu}' = \begin{cases} f_{n-4}-1, \mbox{ if } n\equiv 2 \pmod{3}, \\ 2f_{n-5}-1, \mbox{ if } n\equiv 1 \pmod{3}, \\ f_{n-5}-1, \mbox{ if } n\equiv 0 \pmod{3}.
		\end{cases}\]
		
		If $n\geq 5$ is odd, then
		\[ \overline{\mu}' = \begin{cases} \frac{f_{n-1}+f_{n-3}}{2}-1, \mbox{ if } n\equiv 2 \pmod{3}, \\ f_{n-2}+f_{n-4}-1, \mbox{ if } n\equiv 1 \pmod{3}, \\ \frac{f_{n-2}+f_{n-4}}{2}-1, \mbox{ if } n\equiv 0 \pmod{3}.
		\end{cases}\]
	\end{claim}
	
	With arguments similar to those in Corollary~\ref{cor17tri}, we establish the following result.
	
	\begin{corollary}\label{cor24tri}
		It holds that:
		\begin{itemize}
			\item If $n=8$, then $\overline{\mu}'<\left\lfloor \frac{\ell-\overline{r}}{\overline{u}} \right\rfloor$, and therefore we can apply Theorem~\ref{thm11tri6}(a).
			\item If $n\geq 10$, then $\overline{\mu}'>\left\lfloor \frac{\ell-\overline{r}}{\overline{u}} \right\rfloor$, and therefore we can apply Theorem~\ref{thm11tri6}(b).
			\item If $n=5$, then $\overline{\mu}'<\left\lfloor \frac{\ell-\overline{r}}{\overline{u}} \right\rfloor$, and therefore we can apply Theorem~\ref{thm11tri6}(a).
			\item If $n\geq 7$, then $\overline{\mu}'>\left\lfloor \frac{\ell-\overline{r}}{\overline{u}} \right\rfloor$, and therefore we can apply Theorem~\ref{thm11tri6}(b).
		\end{itemize}
	\end{corollary}
	
	\begin{claim}\label{claim25tri}
		If $n\geq 10$ is even, then the applicable values in Theorem~\ref{thm11tri6}(b) are
		\[ \begin{cases} d_1 = \frac{f_{n-1}}{2},\, d_2 = \frac{f_{n+2}}{2},\, p_1 = 2f_{n-5}-1,\, p_2 = f_{n-5}+f_{n-7}-1,  \mbox{ if } n\equiv 1 \pmod{3}, \\[3pt] d_1 = f_{n+1},\, d_2 = \frac{f_{n+3}}{2},\, p_1 = f_{n-5}-1,\, p_2 = 2f_{n-7}-1, \mbox{ if } n\equiv 0 \pmod{3}, \\[3pt] d_1 = \frac{f_{n+1}}{2},\, d_2 = f_{n+2},\, p_1 = f_{n-4}-1,\, p_2 = f_{n-7}-1, \mbox{ if } n\equiv 2 \pmod{3}.
		\end{cases}\]
		
		If $n\geq 7$ is odd, then the applicable values in Theorem~\ref{thm11tri6}(b) are
		\[ \begin{cases} d_1 = \frac{f_{n-1}}{2},\, d_2 = \frac{f_{n+2}}{2},\, p_1 = f_{n-2}+f_{n-4}-1,\, p_2 = \frac{5}{2}f_{n-4}-1,  \mbox{ if } n\equiv 1 \pmod{3}, \\[3pt] d_1 = f_{n+1},\, d_2 = \frac{f_{n+3}}{2},\, p_1 = \frac{f_{n-2}+f_{n-4}}{2}-1,\, p_2 = f_{n-4}+f_{n-6}-1, \mbox{ if } n\equiv 0 \pmod{3}, \\[3pt] d_1 = \frac{f_{n+1}}{2},\, d_2 = f_{n+2},\, p_1 = \frac{f_{n-1}+f_{n-3}}{2}-1,\, p_2 = \frac{f_{n-4}+f_{n-6}}{2}-1, \mbox{ if } n\equiv 2 \pmod{3}.
		\end{cases}\]
	\end{claim}
	
	\begin{remark}\label{rem26tri}
		We have that $p_1=\overline{\mu}'$.
	\end{remark}
	
	\begin{corollary}\label{cor27tri}
		Regarding Theorem~\ref{thm11tri6}(b), it holds that:
		\begin{enumerate}
			\item If $n=3k_0+1$ with $k_0\in\mathbb{N}\setminus\{0\}$, then
			\begin{itemize}
				\item $y_{p_1}=f_n-4$, $y_{p_2}=\frac{f_n+f_{n-2}}{2}-4$, if $n\geq10$ is even,
				\item $y_{p_1}=f_n-1$, $y_{p_2}=\frac{f_n+f_{n-2}}{2}-1$, if $n\geq7$ is odd,
				\item $\mathrm{PF}(S(n)) = \left\{ \left( \frac{f_{n-1}}{2}-1\right)f_{n+1}^2 + (f_n-1)f_{n+2}^2, \left(\frac{f_{n+2}}{2}-1\right)f_{n+1}^2 + \left(\frac{f_n+f_{n-2}}{2}-1\right)f_{n+2}^2 \right\} - f_n^2$.
			\end{itemize}
			
			\item If $n=3k_0+3$ with $k_0\in\mathbb{N}\setminus\{0\}$, then
			\begin{itemize}
				\item $y_{p_1}=\frac{f_n}{2}-4$, $y_{p_2}=f_{n-2}-4$, if $n\geq12$ is even,
				\item $y_{p_1}=\frac{f_n}{2}-1$, $y_{p_2}=f_{n-2}-1$, if $n\geq9$ is odd,
				\item $\mathrm{PF}(S(n)) = \left\{ (f_{n+1}-1)f_{n+1}^2 + \left(\frac{f_n}{2}-1\right)f_{n+2}^2, \left(\frac{f_{n+3}}{2}-1\right)f_{n+1}^2 + (f_{n-2}-1)f_{n+2}^2 \right\} - f_n^2$.
			\end{itemize}
			
			\item If $n=3k_0+2$ with $k_0\in\mathbb{N}\setminus\{0\}$, then
			\begin{itemize}
				\item $y_{p_1}=\frac{f_{n+1}}{2}-4$, $y_{p_2}=\frac{f_{n-2}}{2}-4$, if $n\geq14$ is even,
				\item $y_{p_1}=\frac{f_{n+1}}{2}-1$, $y_{p_2}=\frac{f_{n-2}}{2}-1$, if $n\geq11$ is odd,
				\item $\mathrm{PF}(S(n)) = \left\{ \left(\frac{f_{n+1}}{2}-1\right)f_{n+1}^2 + \left(\frac{f_{n+1}}{2}-1\right)f_{n+2}^2, (f_{n+2}-1)f_{n+1}^2 + \left(\frac{f_{n-2}}{2}-1\right)f_{n+2}^2 \right\} - f_n^2$.
			\end{itemize}
		\end{enumerate}
		Moreover, in any case, $\mathrm{F}(S(n)) = \max(\mathrm{PF}(S(n)))$.
	\end{corollary}

	\section{Algorithmic process by Rosales and Garc{\'i}a-S{\'a}nchez}\label{rgs}
	
	First, we present the key result for this section.
	
	\begin{lemma}\label{lem02rgs}
		If $n\geq8$, then
		\begin{enumerate}
			\item $f_{n+4}f_n^2 = f_{n+1}f_{n+1}^2 + f_{n-2}f_{n+2}^2$.
			\item $f_{n+2}f_{n+1}^2 = f_{n+2}f_n^2 + f_{n-1}f_{n+2}^2$.
			\item $f_n f_{n+2}^2 = f_{n+3} f_n^2 + f_n f_{n+1}^2$.
		\end{enumerate}
	\end{lemma}
	
	\begin{proof}
		\begin{enumerate}
			\item Since $f_{n+1} f_{n+1}^2 = f_{n+1}^3 = (f_{n+2}-f_n)^3 = f_{n+2}^3 - 3f_{n+2}^2 f_n + 3f_{n+2} f_n^2 - f_n^3 = (3f_{n+2} - f_n)f_n^2 + (f_{n+2}-3f_n)f_{n+2}^2 = f_{n+4}f_n^2 - f_{n-2}f_{n+2}^2$, we deduce the equality.
			
			\item Since $f_{n+1}^2 = (f_{n+2}-f_n)^2 = f_{n+2}^2 - 2f_{n+2} f_n + f_n^2 = (f_{n+2}-2f_n)f_{n+2} + f_n^2 = f_{n-1} f_{n+2} + f_n^2$, then we have that $f_{n+2} f_{n+1}^2 = f_{n+2} f_n^2 + f_{n-1} f_{n+2}^2$.
			
			\item Since $f_{n+2}^2 = (f_{n+1}+f_n)^2 = f_{n+1}^2 + 2f_{n+1}f_n + f_n^2 = f_{n+1}^2 + (2f_{n+1}+f_n)f_n = f_{n+1}^2 + f_{n+3}f_n$, it follows that 
			$f_n f_{n+2}^2 = f_{n+3} f_n^2 + f_n f_{n+1}^2$.
			
			Expression 3 can also be obtained by adding Expressions 1 and 2 and then rearranging the terms of the resulting expression.
		\end{enumerate}
	\end{proof}
	
	As usual, $\#A$ represents the cardinality of the set $A$.
	
	\begin{lemma}\label{lem03rgs}
		$\#\big\{ \{0,\dots,f_{n+2}-1\} \times \{0,\dots,f_n-1\} \setminus \{f_{n+1},\dots,f_{n+2}-1\} \times \{f_{n-2},\dots,f_n-1\} \big\} = 2f_n^2$.
	\end{lemma}
	
	\begin{proof}
		$\#\big\{ \{0,\dots,f_{n+2}-1\} \times \{0,\dots,f_n-1\} \setminus \{f_{n+1},\dots,f_{n+2}-1\} \times \{f_{n-2},\dots,f_n-1\} \big\} = f_{n+2}f_n - \left( f_{n+2}-f_{n+1} \right) \left( f_n-f_{n-2} \right) = f_{n+2}f_n - f_nf_{n-1} = f_n \left( f_{n+2}-f_{n-1} \right) = 2f_n^2$.
	\end{proof}
	
	\begin{remark}\label{rem04rgs}
		As a consequence of Lemma~\ref{lem02rgs}, we have that $\mathrm{Ap}(S(n),f_n^2) \subseteq \big\{ \lambda f_{n+1}^2 + \mu f_{n+2}^2 \mid (\lambda,\mu)\in \{0,\dots,f_{n+2}-1\} \times \{0,\dots,f_n-1\} \setminus \{f_{n+1},\dots,f_{n+2}-1\} \times \{f_{n-2},\dots,f_n-1\} \big\}$. However, the equality is not true as a consequence of Lemma~\ref{lem03rgs}.
	\end{remark}
	
	\begin{example}\label{exmp05rgs}
		For $S(4)=\langle 9,25,64 \rangle$, we have that
		\[ \mathrm{Ap}(S(4),9) = \{ 0,64,128,75,139,50,114,25,89\}. \]
		Moreover, by Remark~\ref{rem04rgs}, it follows that
		\[ \mathrm{Ap}(S(4),9) \subseteq \big\{ \lambda f_{n+1}^2 + \mu f_{n+2}^2 \mid (\lambda,\mu)\in \{(0,0), (0,1), (0,2), (1,0), (1,1), (1,2), (2,0), (2,1), (2,2), \] \[ (3,0), (3,1), (3,2), (4,0), (4,1), (4,2), (5,0), (6,0), (7,0)\} \big\} = \]
		\[ \{ 0, 64, 128, 25, 29,153, 50, 114, 178, 75, 139, 203, 100, 164, 228, 125, 150, 175 \}. \]
		Observe that for each $i\in\{0,1,\dots,8\}$ there appear two elements congruent to $i$ modulo $9$.
	\end{example}
	
	The following result is easy to prove by induction on $k$.
	
	\begin{lemma}\label{lem06rgs}
		If $k\in\mathbb{N}$, then $f_{3k}$ is even, while $f_{3k+1}$ and $f_{3k+2}$ are odd.
	\end{lemma}
	
	To improve Lemma~\ref{lem02rgs}, we analyse three cases that depend on the parity of $f_n$, $f_{n+1}$, and $f_{n+2}$. We begin with the case where $f_n$ is even (and therefore $f_{n+1}$ and $f_{n+2}$ are odd).
	
	\begin{lemma}\label{lem14rgs}
		If $f_n$ is even (that is, $n=3k$ with $k\in\mathbb{N}\setminus\{0,1\}$), then
		\begin{enumerate}
			\item $f_{n+4}f_n^2 = f_{n+1}f_{n+1}^2 + f_{n-2}f_{n+2}^2$.
			\item $\frac{f_{n+3}}{2}f_{n+1}^2 = \frac{f_{n+2}+f_{n+4}}{2} f_n^2 + \frac{f_{n-3}}{2}f_{n+2}^2$.
			\item $\frac{f_n}{2}f_{n+2}^2 = \frac{f_{n+3}}{2}f_n^2 + \frac{f_n}{2}f_{n+1}^2$.
		\end{enumerate}
	\end{lemma}
	
	\begin{proof}
		Expression 1 is Expression 1 of Lemma~\ref{lem02rgs}.
		
		On the other hand, since $f_n$ is even, then$f_{n+3}$ is also even (Lemma~\ref{lem06rgs}). This allows us to deduce Expression 3 from Expression 3 of Lemma~\ref{lem02rgs}.
		
		Finally, considering that $f_{n-3}$ is even and, by Lemma~\ref{lem06rgs}, that $f_{n+2}$ and $f_{n+4}$ are odd (and therefore $f_{n+2}+f_{n+4}$ is even), then (adding Expressions 1 and 3)
		\[ f_{n+4}f_n^2 + \frac{f_n}{2} f_{n+2}^2 = f_{n+1}f_{n+1}^2 + f_{n-2}f_{n+2}^2 + \frac{f_{n+3}}{2} f_n^2 + \frac{f_n}{2} f_{n+1}^2 \Rightarrow \]
		\[ \left(f_{n+4} - \frac{f_{n+3}}{2} \right)f_n^2 + \left(\frac{f_n}{2} - f_{n-2}\right)f_{n+2}^2 = \left(f_{n+1}+\frac{f_n}{2}\right)f_{n+1}^2, \]
		from which follows Expression 2.
	\end{proof}
	
	\begin{proposition}\label{prop15rgs}
		If $f_n$ is even (that is, $n=3k$ with $k\in\mathbb{N}\setminus\{0,1\}$), then
		\[ \mathrm{Ap}(S(n),f_n^2) = \left\{ \lambda f_{n+1}^2 + \mu f_{n+2}^2 \mid (\lambda,\mu)\in C_1 \times C_2 \setminus C_3 \times C_4 \right\}, \]
		where
		\[ C_1=\left\{0,\dots,\frac{f_{n+3}}{2}-1\right\}, \;\;
		C_2=\left\{0,\dots,\frac{f_n}{2}-1\right\}, \]
		\[ C_3=\left\{f_{n+1},\dots,\frac{f_{n+3}}{2}-1\right\}, \;\; C_4=\left\{f_{n-2},\dots,\frac{f_n}{2}-1\right\}. \]
	\end{proposition}
	
	\begin{proof}
		By Lemma~\ref{lem14rgs}, we have that $\mathrm{Ap}(S(n),f_n^2) \subseteq \left\{ \lambda f_{n+1}^2 + \mu f_{n+2}^2 \mid (\lambda,\mu)\in C_1 \times C_2 \setminus C_3 \times C_4 \right\}$.
		
		On the other hand,
		\[ \#\left[ \left\{0,\dots,\frac{f_{n+3}}{2}-1 \right\} \times \left\{0,\dots,\frac{f_n}{2}-1 \right\} \setminus \left\{f_{n+1},\dots,\frac{f_{n+3}}{2}-1 \right\} \times \left\{f_{n-2},\dots,\frac{f_n}{2}-1 \right\} \right] = \]
		\[ \frac{f_{n+3}}{2} \times \frac{f_n}{2} - \left( \frac{f_{n+3}}{2} - f_{n+1} \right) \times \left( \frac{f_n}{2} - f_{n-2} \right) = \frac{f_{n+3}f_n - f_nf_{n-3}}{4} = f_n\frac{f_{n+3} - f_{n-3}}{4} = f_n^2. \]
		Since $\#\mathrm{Ap}(S(n),f_n^2)=f_n^2$, the result is proved.
	\end{proof}
	
	\begin{example}\label{exmp16rgs}
		For $n=6$, we have that $S(6)=\langle 64,169,441 \rangle$ and, by applying Proposition~\ref{prop15rgs},
		\[ \mathrm{Ap}(S(6),64) = \big\{ 169\lambda + 441\mu \mid (\lambda,\mu)\in \{0,\dots,16\} \times \{0,1,2,3\} \setminus \{13,14,15,16\} \times \{3\} \big\}. \]
	\end{example}
	
	Let us now see what happens when $f_{n+1}$ is even (and therefore $f_n$ and $f_{n+2}$ are odd).
	
	\begin{lemma}\label{lem11rgs}
		If $f_{n+1}$ is even (that is, if $n=3k-1$ with $k\in\mathbb{N}\setminus\{0,1\}$), then
		\begin{enumerate}
			\item $\frac{f_{n+4}}{2}f_n^2 = \frac{f_{n+1}}{2}f_{n+1}^2 + \frac{f_{n-2}}{2}f_{n+2}^2$.
			\item $f_{n+2}f_{n+1}^2 = f_{n+2}f_n^2 + f_{n-1}f_{n+2}^2$.
			\item $\frac{f_{n+1}}{2}f_{n+2}^2 = \frac{f_{n+1}}{2}f_n^2 + \frac{f_n+f_{n+2}}{2}f_{n+1}^2$.
		\end{enumerate}
	\end{lemma}
	\begin{proof}
		Expression 2 is Expression 2 of Lemma~\ref{lem02rgs}.
		
		On the other hand, since $f_{n+1}$ is even, then $f_{n-2}$ and $f_{n+4}$ are also even (Lemma~\ref{lem06rgs}). Thereby, we can deduce Expression 1 from Expression 1 of Lemma~\ref{lem02}.
		
		Finally, having in mind that $f_n$ and $f_{n+2}$ are odd (Lemma~\ref{lem06rgs}), then $f_n+f_{n+2}$ is even. Therefore (adding Expressions 1 and 2)
		\[ \frac{f_{n+4}}{2}f_n^2 + f_{n+2}f_{n+1}^2 = \frac{f_{n+1}}{2}f_{n+1}^2 + \frac{f_{n-2}}{2}f_{n+2}^2 + f_{n+2}f_n^2 + f_{n-1}f_{n+2}^2 \Rightarrow \]
		\[ \left(\frac{f_{n+4}}{2} - f_{n+2} \right)f_n^2 + \left(f_{n+2} - \frac{f_{n+1}}{2}\right)f_{n+1}^2 = \left(\frac{f_{n-2}}{2} + f_{n-1}\right)f_{n+2}^2, \]
		from which we deduce Expression 3.
	\end{proof}
	
	\begin{proposition}\label{prop12rgs}
		If $f_{n+1}$ is even (that is, if $n=3k-1$ with $k\in\mathbb{N}\setminus\{0,1\}$), then
		\[ \mathrm{Ap}(S(n),f_n^2) = \left\{ \lambda f_{n+1}^2 + \mu f_{n+2}^2 \mid (\lambda,\mu)\in C_1 \times C_2 \setminus C_3 \times C_4 \right\}, \]
		where
		\[ C_1=\{0,\dots,f_{n+2}-1\}, \;\; C_2=\left\{0,\dots,\frac{f_{n+1}}{2}-1\right\} \]
		\[ C_3=\left\{\frac{f_{n+1}}{2},\dots,f_{n+2}-1\right\}, \;\; C_4=\left\{\frac{f_{n-2}}{2},\dots,\frac{f_{n+1}}{2}-1\right\}. \]
	\end{proposition}
	
	\begin{proof}
		By Lemma~\ref{lem11rgs}, we have that $\mathrm{Ap}(S(n),f_n^2) \subseteq \left\{ \lambda f_{n+1}^2 + \mu f_{n+2}^2 \mid (\lambda,\mu)\in C_1 \times C_2 \setminus C_3 \times C_4 \right\}$.
		
		Since
		\[ \#\left[ \{0,\dots,f_{n+2}-1\} \times \left\{0,\dots,\frac{f_{n+1}}{2}-1\right\} \setminus \left\{\frac{f_{n+1}}{2},\dots,f_{n+2}-1\right\} \times \left\{\frac{f_{n-2}}{2},\dots,\frac{f_{n+1}}{2}-1\right\} \right] = \]
		\[ f_{n+2} \times \frac{f_{n+1}}{2} - \left( f_{n+2} - \frac{f_{n+1}}{2} \right) \times \left( \frac{f_{n+1}}{2} - \frac{f_{n-2}}{2} \right) = \frac{f_{n+2}f_{n-2}}{2} + \frac{f_{n+1}(f_{n+1}-f_{n-2})}{4} = \]
		\[ \frac{f_{n+1}f_{n-2} + f_{n}f_{n-2}}{2} + \frac{f_{n+1}f_{n-1}}{2} = \frac{f_{n+1}f_n + f_{n}f_{n-2}}{2} = f_n\frac{f_{n+1} + f_{n-2}}{2} = f_n^2, \]
		the result is proved.
	\end{proof}
	
	\begin{example}\label{exmp13rgs}
		For $n=5$, we have that $S(5)=\langle 25,64,169 \rangle$ and, by applying Proposition~\ref{prop12rgs}, we conclude that
		\[ \mathrm{Ap}(S(5),25) = \big\{ 64\lambda + 169\mu \mid (\lambda,\mu)\in \{0,\dots,12\} \times \{0,1,2,3\} \setminus \{4,\dots,12\} \times \{1,2,3\} \big\}. \]
	\end{example}
	
	Finally, we show the case when $f_{n+2}$ is even (and therefore $f_n$ and $f_{n+1}$ are odd).
	
	\begin{lemma}\label{lem07rgs}
		If $f_{n+2}$ is even (that is, if $n=3k-2$ with $k\in\mathbb{N}\setminus\{0,1\}$), then
		\begin{enumerate}
			\item $\frac{f_{n+5}}{2}f_n^2 = \frac{f_{n-1}}{2}f_{n+1}^2 + \frac{f_{n-2}+f_n}{2}f_{n+2}^2$.
			\item $\frac{f_{n+2}}{2}f_{n+1}^2 = \frac{f_{n+2}}{2}f_n^2 + \frac{f_{n-1}}{2}f_{n+2}^2$.
			\item $f_nf_{n+2}^2 = f_{n+3}f_n^2 + f_nf_{n+1}^2$.
		\end{enumerate}
	\end{lemma}
	\begin{proof}
		Expression 3 is Expression 3 of Lemma~\ref{lem02rgs}.
		
		On the other hand, since $f_{n+2}$ is even, then $f_{n-1}$ is also even (Lemma~\ref{lem06rgs}). Therefore, we deduce Expression 2 from Expression 2 of Lemma~\ref{lem02rgs}.
		
		At last, since $f_{n+5}$ is even and $f_{n-2}$ are $f_n$ odd (Lemma~\ref{lem06rgs}) (and, consequently, $f_{n-2}+f_n$ is even), then (adding Expressions 2 and 3)
		\[ \frac{f_{n+2}}{2}f_{n+1}^2 + f_nf_{n+2}^2 = \frac{f_{n+2}}{2}f_n^2 + \frac{f_{n-1}}{2}f_{n+2}^2 + f_{n+3}f_n^2 + f_nf_{n+1}^2 \Rightarrow \]
		\[ \left(\frac{f_{n+2}}{2} - f_n \right)f_{n+1}^2 + \left(f_n - \frac{f_{n-1}}{2}\right)f_{n+2}^2 = \left(\frac{f_{n+2}}{2} + f_{n+3}\right)f_n^2, \]
		from we get Expression 1.
	\end{proof}
	
	\begin{proposition}\label{prop08rgs}
		If $f_{n+2}$ is even (that is, if $n=3k-2$ with $k\in\mathbb{N}\setminus\{0,1\}$), then
		\[ \mathrm{Ap}(S(n),f_n^2) = \left\{ \lambda f_{n+1}^2 + \mu f_{n+2}^2 \mid (\lambda,\mu)\in C_1 \times C_2 \setminus C_3 \times C_4 \right\}, \]
		where
		\[ C_1=\left\{0,\dots,\frac{f_{n+2}}{2}-1\right\}, \;\; C_2=\{0,\dots,f_n-1\}, \]
		\[ C_3=\left\{\frac{f_{n-1}}{2},\dots,\frac{f_{n+2}}{2}-1\right\}, \;\; C_4=\left\{\frac{f_{n-2}+f_n}{2},\dots,f_n-1\right\}. \]
	\end{proposition}
	
	\begin{proof}
		By Lemma~\ref{lem07rgs}, we have that $\mathrm{Ap}(S(n),f_n^2) \subseteq \left\{ \lambda f_{n+1}^2 + \mu f_{n+2}^2 \mid (\lambda,\mu)\in C_1 \times C_2 \setminus C_3 \times C_4 \right\}$.
		
		Since
		\[ \#\left[ \left\{0,\dots,\frac{f_{n+2}}{2}-1\right\} \times \{0,\dots,f_n-1\} \setminus \left\{\frac{f_{n-1}}{2},\dots,\frac{f_{n+2}}{2}-1\right\} \times \left\{\frac{f_{n-2}+f_n}{2},\dots,f_n-1\right\} \right] = \]
		\[ \frac{f_{n+2}}{2} \times f_n - \left( \frac{f_{n+2}}{2} - \frac{f_{n-1}}{2} \right) \times \left( f_n - \frac{f_{n-2}+f_n}{2} \right) = \frac{f_{n+2}}{2}f_n - f_n\frac{f_n-f_{n-2}}{2} = f_n\frac{f_{n+2} - f_{n-1}}{2} = f_n^2, \]
		we have the result.
	\end{proof}
	
	\begin{example}\label{exmp09rgs}
		For $n=4$, we have that $S(4)=\langle 9,25,64 \rangle$ and, by applying Proposition~\ref{prop08rgs},
		\[ \mathrm{Ap}(S(4),9) = \big\{ 25\lambda + 64\mu \mid (\lambda,\mu)\in \{0,1,2,3\} \times \{0,1,2\} \setminus \{1,2,3\} \times \{2\} \big\}. \]
	\end{example}
	
	By Proposition~\ref{prop15rgs}, \ref{prop12rgs}, and \ref{prop08rgs}, we deduce the main result of this section.
	
	\begin{theorem}\label{thm17rgs}
		Let $n\geq4$.
		\begin{enumerate}
			\item If $f_n$ is even (that is, if $n=3k$ with $k\in\mathbb{N}\setminus\{0,1\}$), then $\mathrm{PF}(S(n)) =$
			\[ \left\{ (f_{n+1}-1)f_{n+1}^2 + \left(\frac{f_n}{2}-1\right)f_{n+2}^2 - f_n^2, \left(\frac{f_{n+3}}{2}-1\right)f_{n+1}^2 + (f_{n-2}-1)f_{n+2}^2 - f_n^2 \right\}. \]
			\item If $f_{n+1}$ is even (that is, if $n=3k-1$ with $k\in\mathbb{N}\setminus\{0,1\}$), then $\mathrm{PF}(S(n)) =$
			\[ \left\{ \left(\frac{f_{n+1}}{2}-1\right)f_{n+1}^2 + \left(\frac{f_{n+1}}{2}-1\right)f_{n+2}^2 - f_n^2, (f_{n+2}-1)f_{n+1}^2 + \left(\frac{f_{n-2}}{2}-1\right)f_{n+2}^2 - f_n^2 \right\}. \]
			\item If $f_{n+2}$ is even (that is, if $n=3k-2$ with $k\in\mathbb{N}\setminus\{0,1\}$), then $\mathrm{PF}(S(n)) =$
			\[ \left\{ \left(\frac{f_{n-1}}{2}-1\right)f_{n+1}^2 + (f_n-1)f_{n+2}^2 - f_n^2, \left(\frac{f_{n+2}}{2}-1\right)f_{n+1}^2 + \left(\frac{f_{n-2}+f_n}{2}-1\right)f_{n+2}^2 - f_n^2 \right\}. \]
		\end{enumerate}
	\end{theorem}
	
	The solution to the proposed Frobenius problem is derived directly from the above theorem.
	
	\begin{corollary}\label{cor18rgs}
		Let $n\geq4$.
		\begin{enumerate}
			\item If $f_n$ is even (that is, if $n=3k$ with $k\in\mathbb{N}\setminus\{0,1\}$), then
			\[ \mathrm{F}(S(n)) = \left(\frac{f_{n+3}}{2}-1\right)f_{n+1}^2 + (f_{n-2}-1)f_{n+2}^2 - f_n^2. \]
			\item If $f_{n+1}$ is even (that is, if $n=3k-1$ with $k\in\mathbb{N}\setminus\{0,1\}$), then
			\[ \mathrm{F}(S(n)) = (f_{n+2}-1)f_{n+1}^2 + \left(\frac{f_{n-2}}{2}-1\right)f_{n+2}^2 - f_n^2. \]
			\item If $f_{n+2}$ is even (that is, if $n=3k-2$ with $k\in\mathbb{N}\setminus\{0,1\}$), then
			\[ \mathrm{F}(S(n)) = \left(\frac{f_{n+2}}{2}-1\right)f_{n+1}^2 + \left(\frac{f_{n-2}+f_n}{2}-1\right)f_{n+2}^2 - f_n^2. \]
		\end{enumerate}
	\end{corollary}
	
	\begin{proof}
		\begin{enumerate}
			\item We have that
			\[ (f_{n+1}-1)f_{n+1}^2 + \left(\frac{f_n}{2}-1\right)f_{n+2}^2 - f_n^2 < \left(\frac{f_{n+3}}{2}-1\right)f_{n+1}^2 + (f_{n-2}-1)f_{n+2}^2 -f_n^2 \Longleftrightarrow \]
			\[ 2f_{n+1}f_{n+1}^2 + f_nf_{n+2}^2 < f_{n+3}f_{n+1}^2 + 2f_{n-2}f_{n+2}^2 \Longleftrightarrow f_{n-3}f_{n+2}^2 < f_nf_{n+1}^2 \Longleftrightarrow \]
			\[ 2f_{n-3}f_nf_{n+1} + f_{n-3}f_n^2 < 2f_{n-2}f_{n+1}^2 \Longleftrightarrow f_{n-3}f_n^2 < 2(f_{n-4}f_n+f_{n-2}f_{n-1})f_{n+1}. \]
			Since $f_n < 2f_{n-1}$ for all $n\geq 4$, then $f_{n-3}\cdot f_n\cdot f_n < f_{n-2}\cdot 2f_{n-1}\cdot f_{n+1}$ and the result is proved.
			
			\item We have that
			\[ \left(\frac{f_{n+1}}{2}-1\right)f_{n+1}^2 + \left(\frac{f_{n+1}}{2}-1\right)f_{n+2}^2 - f_n^2 < (f_{n+2}-1)f_{n+1}^2 + \left(\frac{f_{n-2}}{2}-1\right)f_{n+2}^2 - f_n^2 \Longleftrightarrow \]
			\[ f_{n+1}f_{n+1}^2 + f_{n+1}f_{n+2}^2 < 2f_{n+2}f_{n+1}^2 + f_{n-2}f_{n+2}^2 \Longleftrightarrow f_{n+1}^3 < f_{n-1}f_{n+1}f_{n+2} + f_{n-2}f_{n+2}^2 \Longleftrightarrow \]
			\[ 2f_{n-1}f_{n+1}^2 < f_{n-1}f_{n+1}f_{n+2} + 2f_{n-2}f_nf_{n+1} + f_{n-2}f_n^2 \Longleftrightarrow f_{n-1}^2f_{n+1} < 2f_{n-2}f_nf_{n+1} + f_{n-2}f_n^2. \]
			Since $f_{n-1} < 2f_{n-2}$ for all $n\geq 4$, then $f_{n-1}\cdot f_{n-1}\cdot f_{n+1} < 2f_{n-2}\cdot f_n\cdot f_{n+1}$ and the result is proved.
			
			\item We have that
			\[ \left(\frac{f_{n-1}}{2}-1\right)f_{n+1}^2 + (f_n-1)f_{n+2}^2 - f_n^2 < \left(\frac{f_{n+2}}{2}-1\right)f_{n+1}^2 + \left(\frac{f_{n-2}+f_n}{2}-1\right)f_{n+2}^2 - f_n^2 \Longleftrightarrow \]
			\[ f_{n-1}f_{n+1}^2 + 2f_nf_{n+2} < f_{n+2}f_{n+1}^2 + (f_{n-2}+f_n)f_{n+2}^2 \Longleftrightarrow f_{n-1}f_{n+2}^2 < 2f_nf_{n+1}^2 \Longleftrightarrow \]
			\[ f_{n-1}f_{n+1}^2 + f_{n-1}f_n^2 < 2f_n^2f_{n+1} \Longleftrightarrow 2f_{n-1}^2f_n + f_{n-1}^3 < 2f_n^3 \Longleftrightarrow f_{n-1}^3 < 2f_{n-2}f_nf_{n+1}. \]
			Since $f_{n-1} < 2f_{n-2}$ for all $n\geq 4$, then $f_{n-1}\cdot f_{n-1}\cdot f_{n-1} < 2f_{n-2}\cdot f_n\cdot f_{n+1}$ and the result is proved.
		\end{enumerate}
	\end{proof}
	
	We finish with an illustrative example.
	
	\begin{example}\label{exmp19rgs}
		\begin{enumerate}
			\item $S(4)=\langle 9,25,64 \rangle$, $\mathrm{PF}(S(4))=\{119,130\}$, and $\mathrm{F}(S(4))=130$.			
			\item $S(5)=\langle 25,64,169 \rangle$, $\mathrm{PF}(S(5))=\{674,743\}$, and $\mathrm{F}(S(5))=743$.
			\item $S(6)=\langle 64,169,441 \rangle$, $\mathrm{PF}(S(6))=\{3287,3522\}$, and $\mathrm{F}(S(6))=3522$.
		\end{enumerate}
	\end{example}

	\section{Conclusions}\label{conclusions}
	
	After reviewing Sections \ref{rar}, \ref{tri}, and \ref{rgs}, we believe that the algorithmic process proposed by Rosales and Garc{\'i}a-S{\'a}nchez is more convenient than those by Ram{\'i}rez Alfons{\'i}n and R{\o}dseth or Tripathy for finding a general formula to solve the Frobenius problem for a family of numerical semigroups. In fact, all steps in Section \ref{rgs} have been rigorously proven. However, in Sections \ref{rar} and \ref{tri}, we have not been able to demonstrate all the results. Furthermore, we even had to propose alternative statements for some results in Subsection \ref{tri-pre} because several errors were detected in the original paper (\cite{tripathi}). These results were also mentioned, without proofs, in \cite{suhajda}.
	
	As future work, we propose to provide rigorous proofs for all claims in Section \ref{rar} (we suspect this will require using continued fractions and difference equations) and Subsection \ref{tri-app} (currently, we have no ideas about possible tools to use). Additionally, we aim to prove the alternative statements of the theorems in Subsection \ref{tri-app}, for which we also lack immediate ideas on the necessary tools.

	

	\section*{Declarations}
		
	\begin{itemize}
		\item \textbf{Funding.} The authors are supported by Proyecto de Excelencia de la Junta de Andaluc{\'i}a (ProyExcel 00868) and Junta de Andaluc{\'i}a Grant Number FQM-343.
		\item \textbf{Conflict of interest/Competing interests.} The authors declare that they have no conflicts of interest.
		\item \textbf{Data availability.} Data sharing does not apply to this article as no datasets were generated or analysed during the current study.
		\item \textbf{Author contribution.} The authors contributed equally to this work.
	\end{itemize}

\end{document}